\documentclass[12pt,reqno]{amsproc}
\usepackage{amsfonts}
\usepackage{amsmath}
\usepackage{amssymb}
\usepackage{graphicx}

 \theoremstyle{plain}
\newtheorem{theorem}{Theorem}[section]

\newtheorem{corollary}[theorem]{Corollary}

\newtheorem{definition}[theorem]{Definition}
\newtheorem{example}[theorem]{Example}

\newtheorem{lemma}[theorem]{Lemma}

\newtheorem{proposition}[theorem]{Proposition}

\numberwithin{equation}  {section}

\begin{document}

\begin{abstract}
In 2015, Yanni Chen, Don Hadwin and Junhao Shen proved a noncommutative
version of Beurling's theorems for a continuous unitarily invariant norm $%
\alpha $ on a tracial von Neumann algebra $\left( \mathcal{M},\tau \right) $
where $\alpha $ is $\left\Vert \cdot \right\Vert _{1}$-dominating with
respect to $\tau $. In the paper, we first define a class of norms $%
N_{\Delta }\left( \mathcal{M},\tau \right) $ on $\mathcal{M}$, called
determinant, normalized, unitarily invariant continuous norms on $\mathcal{M}$. If $\alpha \in N_{\Delta }\left( 
\mathcal{M},\tau \right) $, then there exists a faithful normal tracial
state $\rho $ on $\mathcal{M}$ such that $\rho \left( x\right) =\tau \left(
xg\right) $ for some positive $g\in L^{1}\left( \mathcal{Z},\tau \right) $
and the determinant of $g$ is positive. For every $\alpha \in N_{\Delta }\left( 
\mathcal{M},\tau \right) $, we study the noncommutative Hardy spaces $%
H^{\alpha }\left( \mathcal{M},\tau \right) $, then prove that the
Chen-Hadwin-Shen theorem holds for $L^{\alpha }\left( \mathcal{M},\tau
\right) $. The key ingredients in the proof of our result include a
factorization theorem and a density theorem for $L^{\alpha }\left( 
\mathcal{M},\rho \right) $.
\end{abstract}

\author{Haihui Fan}
\curraddr{University of New Hampshire}
\email{hun4@wildcats.unh.edu}
\author{Don Hadwin}
\curraddr{University of New Hampshire}
\email{don@unh.edu}
\author{Wenjing Liu}
\address{University of New Hampshire}
\email{wbs4@wildcats.unh.edu}
\title[New Chen-Beurling Theorem]{An Extension of the
Beurling-Chen-Hadwin-Shen Theorem for Noncommutative Hardy Spaces Associated
with Finite von Neumann Algebras}
\maketitle

\section{Introduction}

It has long been of great importance to operator theorist and operator
algebraist to study noncommutative Beurling's theorem\cite{Arveson},\cite%
{B.L.},\cite{B.L.2},\cite{Chen},\cite{H.S.},\cite{Nelson}. We recall some
concepts in noncommutative Hardy spaces with finite von Neumann algebras.
Given a finite von Neumann algebra $\mathcal{M}$ acting on a Hilbert space $H
$, the set of possibly unbounded closed and densely defined operators on $H$
which are affiliated to $\mathcal{M}$, form a topological algebra where the
topology is the (noncommutative) topology of convergence in measure. We
denote this algebra by $\widetilde{\mathcal{M}}$. The trace $\tau $ extends
naturally to the positive operators in $\widetilde{\mathcal{M}}$. The
important fact regarding this algebra, is that it is large enough to
accommodate all the noncommutative $L^{p}$ spaces corresponding to $\mathcal{%
M}$. Specifically, if $1\leq p<\infty $, then we define the space $L^{p}({%
\mathcal{M},\tau )}=\{x\in \widetilde{\mathcal{M}}:\tau (|x|^{p})<\infty \}$,
where the ambient norm is given by $\Vert \cdot \Vert _{p}=\tau (\Vert \cdot
\Vert ^{p})^{1/p}$. The space $L^{\infty }({\mathcal{M},\tau )}$ is defined
to be $\mathcal{M}$ itself. These spaces capture all the usual properties of 
$L^{p}$ spaces, with the dual action of $L^{p}$ on $L^{q}$ ($q$ conjugate to 
$p$) given by $(a,b)\rightarrow \tau (ab)$. For any subset $S$ of $\mathcal{M%
}$, we write $[\mathcal{S}]_{p}$ for the p-norm closure of $\mathcal{S}$ in $%
L^{p}({\mathcal{M},\tau )}$, with the understanding that $[\mathcal{S}]_{p}$
will denote the weak* closure in the case $p=\infty $. W. Arveson \cite%
{Beurling} introduced a concept of maximal subdiagonal algebra in 1967, also
known as a noncommutative $H^{\infty }$ space, to study the analyticity in
operator algebras. Let $\mathcal{M}$ be a finite von Neumann algebra with a
faithful normal tracial state $\tau .$ Let $\mathcal{A}$ be a weak* closed
unital subalgebra of $\mathcal{M},$ and $\mathcal{A} $ is called a finite maximal
subalgebra of $\mathcal{M}$ with respect to $\Phi $ if (i) $\mathcal{A}+%
\mathcal{A}^{\ast }$ is weak* dense in $\mathcal{M}$; (ii) $\Phi (xy)=\Phi
(x)\Phi (y)$ for $\mathcal{8}x,y\in \mathcal{A}$; (iii) $\tau \circ \Phi
=\tau $; and (iv) $\mathcal{D=A}\cap \mathcal{A}^{\ast }.$ Such a finite
maximal subdiagonal subalgebra $\mathcal{A}$ of $\mathcal{M}$ is also called
an $H^{\infty }$ space of $\mathcal{M}.$ For each $1\leq p\leq \infty$,
let $H^{p}$ be the completion of Arveson's noncommutative $H^{\infty }$ with
respect to $\left\Vert \cdot \right\Vert _{p}$. After Arveson's introduction
of noncommutative $H^{p}$ spaces, there are many studies to obtain a
Beurling's theorem for invariant subspaces in noncommutative $H^{p}$ spaces
(for example, see \cite{Bekjan},\cite{B.L.2},\cite{Chen}).

Y. Chen, D. Hadwin, J. Shen obtained a version of the
Blecher-Labuschagne-Beurling invariant subspace theorem on $H^{\infty }$%
-right invariant subspace in a noncommutative $L^{\alpha }(\mathcal{M},\tau
) $ space, where $\alpha $ is a normalized unitarily invariant, $\left\Vert
\cdot \right\Vert _{1}$-dominating, continuous norm.

In this paper, we will extend Chen-Hadwin-Shen's result in \cite{Chen} by
considering drop the condition that $\alpha $ is $\left\Vert \cdot
\right\Vert _{1}$-dominating. By defining a generalized $\alpha $ norm, we
have a version of Chen-Hadwin-Shen's result for noncommutative Hardy spaces.

{\footnotesize THEOREM} \ref{main}. \emph{Let }$\mathcal{M}$\emph{\ be a
finite von Neumann algebra with a faithful, normal, tracial state }$\tau $%
\emph{\ and }$\alpha $\emph{\ be a determinant, normalized, unitarily
invariant, continuous norm on }$\mathcal{M}$\emph{. Then there exists a
faithful normal tracial state }$\rho $\emph{\ on }$\mathcal{M}$ \emph{such
that }$\alpha \in N_{1}\left( \mathcal{M},\rho \right) $\emph{. Let }$%
H^{\infty }$\emph{\ be a finite subdiagonal subalgebra of }$\mathcal{M}$%
\emph{\ and }$\mathcal{D}=H^{\infty }\cap (H^{\infty })^{\ast }$\emph{. If }$%
\mathcal{W}$\emph{\ is a closed subspace of }$L^{\alpha }(\mathcal{M},\tau )$%
\emph{\ such that }$\mathcal{W}H^{\infty }\subseteq \mathcal{W}$\emph{, then
there exists a closed subspace }$\mathcal{Y}$\emph{\ of }$L^{\alpha }(%
\mathcal{M},\tau )$\emph{\ and a family }$\{u_{\lambda }\}_{\lambda \in
\Lambda }$\emph{\ of partial isometries in }$\mathcal{M}$\emph{\ such that:%
\newline
(1) }$u_{\lambda }^{\ast }\mathcal{Y}=0$\emph{\ for all }$\lambda \in
\Lambda $\emph{,\newline
(2) }$u_{\lambda }^{\ast }u_{\lambda }\in \mathcal{D}$\emph{\ and }$%
u_{\lambda }^{\ast }u_{\mu }=0$\emph{\ for all }$\lambda ,\mu \in \Lambda $%
\emph{\ with }$\lambda \neq \mu $\emph{\newline
(3) }$\mathcal{Y}=[H_{0}^{\infty }\mathcal{Y}]_{\alpha }$\emph{\newline
(4) } $\mathcal{W} = \mathcal{Y} \oplus^{col}(\oplus _{\lambda \in
\Lambda }^{col}u_{\lambda }H^{\alpha })$

Many tools used in \cite{Chen} are no longer available in an arbitrary $%
L^{\alpha }(\mathcal{M},\tau )$ space and new techniques  need to be invented.
First, we need using the Fuglede-Kadison determinant$,$ and inner, outer
factorization for noncommutative Hardy space in \cite{Bekjan}. Let $\Delta $
be Fuglede-Kadison determinant on $\mathcal{M}$ defined by\newline
\begin{equation*}
\Delta (x)=exp(\tau (\text{log}|x|))=exp(\int_{0}^{\infty }\text{log}(t)d\nu
_{|x|}(t)),
\end{equation*}%
where $d\nu _{|x|}(t)$ denotes the probability measure on $\mathbb{R_{+}}$,
Also, the definition of this determinant can be extended to a very large class of elements of $\widetilde{\mathcal{M}}.$\\

\begin{definition}
Let $1\leq p\leq \infty $. An element $x\in H^{p}(\mathcal{M},\tau )$ is
outer if $I\in \lbrack xH^{p}(\mathcal{M},\tau )]_{p}$, and $x\in H^{p}(%
\mathcal{M},\tau )$ is strongly outer if $x$ is outer and $\Delta (x)>0$. An
element $u$ is inner if $u\in H^{\infty }(\mathcal{M},\tau )$ and $u$ is
unitary.
\end{definition}

In order to prove our main result of the paper, we first get the following theorem.

{\footnotesize THEOREM} \ref{cdominating}. Let $\mathcal{M}$ be a finite von
Neumann algebra with a faithful, normal, tracial state $\tau $ and $\alpha $
be a normalized, unitarily invariant, continuous norm on $(\mathcal{M},\tau )
$. Then there exists a positive $g\in L^{1}(\mathcal{Z},\tau )$ such that
(i) $\rho (\cdot )=\tau (\cdot g)$ is a faithful normal tracial state on $%
\mathcal{M}$, (ii) $\alpha $ is $c\left\Vert \cdot \right\Vert _{1,\rho }$%
-dominating, for some $c>0$. (iii), $\rho (x)=\tau (xg)$ for every $x\in L^{1}(\mathcal{M},\rho ).$

{\footnotesize THEOREM} \ref{Halpha}. If $\alpha \in N_{\Delta }(\mathcal{M}%
,\tau ),$ then there exists a faithful normal tracial state $\rho $ such
that $H^{\alpha }(\mathcal{M},\rho )=H^{1}(\mathcal{M},\rho )\cap L^{\alpha
}(\mathcal{M},\rho )$.

Then we get a factorization theorem  and a density theorem for $L^{\alpha }(%
\mathcal{M},\tau )$ to get the main theorem.

{\footnotesize THEOREM} \ref{factorization}. Suppose $\alpha \in N_{\Delta
}\left( \mathcal{M},\tau \right) ,$ there exists a faithful normal tracial
state $\rho $ on $\mathcal{M}$ such that $\rho \left( x\right) =\tau \left(
xg\right) $ for some positive $g\in L^{1}\left( \mathcal{Z},\tau \right) $
and the determinant of $g$ is positive. If $x\in \mathcal{M}$ and $x^{-1}\in
L^{\alpha }\left( \mathcal{M},\rho \right) ,$ then there are unitary
operators $u_{1},u_{2}\in \mathcal{M}$ and $s_{1},s_{2}\in H^{\infty }$ such
that $x=u_{1}s_{1}=s_{2}u_{2}$ and $s_{1}^{-1},s_{2}^{-1}\in H^{\alpha }(%
\mathcal{M},\rho )$.

{\footnotesize THEOREM} \ref{density}. Let $\alpha \in N_{\Delta }\left( 
\mathcal{M},\tau \right) $, then there exists a faithful normal tracial
state $\rho $ on $\mathcal{M}$ such that $\rho \left( x\right) =\tau \left(
xg\right) $ for some positive $g\in L^{1}\left( \mathcal{Z},\tau \right) $
and the determinant of $g$ is positive. Also, if $\mathcal{W}$ is a closed
subspace of $L^{\alpha }(\mathcal{M},\rho )$ and $\mathcal{N}$ is a weak*
closed linear subspace of $\mathcal{M}$ such that $\mathcal{W}H^{\infty
}\subset \mathcal{W}$ and $\mathcal{N}H^{\infty }\subset \mathcal{N}$, then%
\newline
(1) $\mathcal{N}=[\mathcal{N}]_{\alpha }\cap \mathcal{M}$,\newline
(2) $\mathcal{W}\cap \mathcal{M}$ is weak* closed in $\mathcal{M}$,\newline
(3) $\mathcal{W}=[\mathcal{W}\cap \mathcal{M}]_{\alpha }$,\newline
(4) If $\mathcal{S}$ is a subspace of $\mathcal{M}$ such that $\mathcal{S}%
H^{\infty }\subset \mathcal{S}$, then $[\mathcal{S}]_{\alpha }=[\overline{%
\mathcal{S}}^{w\ast }]_{\alpha }$, where $\overline{\mathcal{S}}^{w\ast }$
is the weak*-closure of $\mathcal{S}$ in $\mathcal{M}$.

The organization of the paper is as follows. In section 2, we introduce
determinant, normalized, unitarily invariant continuous norms. In section 3,
we study the relations between noncommutative Hardy spaces $H^{\alpha }(%
\mathcal{M},\rho )$ and $H^{\alpha }\left( \mathcal{M},\tau \right) .$ In
section 4, we prove the main result of the paper, a version of
Chen-Hadwin-Shen's result for noncommutative Hardy spaces associated with
new norm. In section 5, we get a generalized noncommutative Beurling's
theorem for special von Neumann algebras.

\section{Determinant, normalized, unitarily invariant continuous norms}
In the following section, we give the property of central valued traces in \cite{Kadison}, and introduce a class of determinant, normalized, unitarily invariant continuous norms on finite von Neumann algebras and some interesting examples from this class. In the end of this section, we will obtain our first theorem. 

\begin{proposition}
\label{phiproperty}If $\mathcal{M}$ is a finite von Neumann algebra with the
center $\mathcal{Z}$ of $\mathcal{M}$, then there is a unique positive linear
mapping $\varphi $ from $\mathcal{M}$ into $\mathcal{Z}$ such that\newline
(1) $\varphi (xy)=\varphi (yx)$ for each $x$ and $y$ in $\mathcal{M}$,%
\newline
(2) $\varphi (z)=z$ for each $z$ in $\mathcal{Z}$,\newline
(3) $\varphi (x)>0$ if $x>0$ for $x$ in $\mathcal{M}$,\newline
(4) $\varphi (zx)=z\varphi (x)$ for each $z$ in $\mathcal{Z}$ and $x$ in $%
\mathcal{M}$,\newline
(5) $\Vert \varphi (x)\Vert \leq \Vert x\Vert $ for $x$ in $\mathcal{M}$,%
\newline
(6) $\varphi $ is ultraweakly continuous,\newline
(7) For any $x\in \mathcal{M}$, $\varphi (x)$ is the unique central element in
the norm closure of the convex hull of $\{uxu^{\ast }|u\in \mathcal{U}(%
\mathcal{M})\}$,\newline
(8) Every tracial state on $\mathcal{M}$ is of the form $\tau \circ \varphi $
where $\tau $ is a state on  $\mathcal{Z}$, i.e. every state on
the center $\mathcal{Z}$ of $\mathcal{M}$ extends uniquely to a tracial
state on $\mathcal{M}$,\newline
(9) $\varphi $ is faithful.
\end{proposition}

Let $\mathcal{M}$ be a finite von Neumann algebra with a faithful, normal,
tracial state $\tau $, the $\left\Vert \cdot \right\Vert _{p}$ is a
mapping from $\mathcal{M}$ to $[0,\infty )$ defined by $\left\Vert
x\right\Vert _{p}=\left( \tau (\left\vert x\right\vert ^{p})\right) ^{1/p}$, 
$\mathcal{8}x\in \mathcal{M},0<p<\infty .$ It is known that $\left\Vert
\cdot \right\Vert _{p}$ is a norm if $1\leq p<\infty $, and a quasi-norm if $%
0<p<1$.

\begin{definition}
Let $\mathcal{M}$ be a finite von Neumann algebra with a faithful, normal,
tracial state $\tau $. Assume $\alpha :\mathcal{M}\rightarrow \lbrack
0,\infty )$ is a norm satisfying\newline
(1) $\alpha \left( I\right) =1$, i.e., $\alpha $ is normalized,\newline
(2) $\alpha \left( x\right) =\alpha \left( \left\vert x\right\vert \right) $%
\ for all $x\in \mathcal{M}$ and $|x|=(x^{\ast }x)^{1/2}$, i.e., $\alpha $ is a
gauge,\newline
(3) $\alpha \left( u^{\ast }xu\right) =\alpha \left( x\right) ,u\in \mathcal{%
U}(\mathcal{M})$ and $x\in \mathcal{M}$, i.e., $\alpha $ is unitarily invariant,%
\newline
(4) $\lim_{\tau \left( e\right) \rightarrow 0}\alpha \left( e\right) =0$\ as 
$e$\ ranges over the projections in $\mathcal{M}$.\emph{\ }i.e., if $\{e_{\lambda}\}$ is a net of projections in $\mathcal{M}$ and $\tau \left( e_{\lambda} \right) \rightarrow 0$, then $\alpha \left( e_{\lambda} \right) \rightarrow 0$  which means $\alpha $ is continuous.\newline
Then we call $\alpha $ is a normalized unitarily invariant continuous norm.
And we denote $N(\mathcal{M},\tau )$ to be the collection of all such norms.
\end{definition}

\begin{definition}
We denote by $N_{1}\left( \mathcal{M},\tau \right) ,$ the collection of all
these norms $\alpha :\mathcal{M}\rightarrow \lbrack 0,\infty )$ such that%
\newline
(1) $\alpha \in N\left( \mathcal{M},\tau \right) $,\newline
(2) $\forall x\in \mathcal{M}$, $\alpha \left( x\right) \geq c\left\Vert
x\right\Vert _{1}$, for some $c>0.$\newline
A norm $\alpha $ in $N_{1}\left( \mathcal{M},\tau \right) $ is called a 
\emph{normalized, unitarily invariant }$\left\Vert \cdot \right\Vert _{1}$%
\emph{-dominating continuous }norm on $\mathcal{M}.$
\end{definition}

\begin{definition}
\label{defndelta}We denote by $N_{\Delta }\left( \mathcal{M},\tau \right) $,
the collection of all these norms $\alpha :\mathcal{M}\rightarrow \lbrack
0,\infty )$ such that\newline
(1) $\alpha \in N\left( \mathcal{M},\tau \right) $,\newline
(2) There exists a positive $g\in L^{1}\left( \mathcal{M},\tau \right) $
such that $\Delta \left( g\right) >0$ and $\alpha \left( x\right) \geq c\tau
\left( \left\vert x\right\vert g\right) $ for some $c>0$.\newline
A norm $\alpha $ in $N_{\Delta }\left( \mathcal{M},\tau \right) $ is called
a \emph{determinant, normalized, unitarily invariant continuous }norm on $%
\mathcal{M}.$
\end{definition}

\begin{example}
For the definition \ref{defndelta}, if we take $g=1,$ then $\alpha \in
N_{1}\left( \mathcal{M},\tau \right) $, i.e., $N_{\Delta }\left( \mathcal{M}%
,\tau \right) \subset N_{1}\left( \mathcal{M},\tau \right) $.
\end{example}

\begin{example}
Each $p$-norm $\left\Vert \cdot \right\Vert _{p}$ is in $N(\mathcal{M},\tau
),$ $N_{1}(\mathcal{M},\tau )$, and $N_{\Delta }(\mathcal{M},\tau )$ for $%
1\leq p<\infty $.
\end{example}

\begin{example}
Let $\mathcal{M}$ be a finite von Neumann algebra with a faithful, normal,
tracial state $\tau $. Let $E(0,1)$ be a symmetric Banach function space on $%
(0,1)$ and $E(\tau )$ be the noncommutative Banach function space with a
norm $\Vert \cdot \Vert _{E(\tau )}$ corresponding to $E(0,1)$ and
associated with $(\mathcal{M},\tau )$. If $E(0,1)$ is also order continuous,
then the restriction of the norm $\Vert \cdot \Vert _{E(\tau )}$ to $%
\mathcal{M}$ lies in $N(\mathcal{M},\tau )$ and $N_{1}(\mathcal{M},\tau )$.
\end{example}

In order to prove the first theorem in this paper, we need the following
lemmas, the first lemma is proved by H. Fan, D. Hadwin and W. Liu in \cite%
{F.H.W}.

\begin{lemma}
\label{lmforcdominating}Suppose $\left( X,\Sigma ,\mu \right) $ is a
probability space and $\alpha $ is a continuous normalized gauge norm on $%
L^{\infty }(\mu )$. Then there exists $0<c<1$ and a probability measure $%
\lambda $ on $\Sigma $ such that $\lambda \ll \mu $ and $\mu \ll \lambda $,
such that $\alpha $ is $c\Vert \cdot \Vert _{1,\lambda }$-dominating.
\end{lemma}

\begin{lemma}
Let $\mathcal{M}$ be a finite von Neumann algebra with a faithful, normal,
tracial state $\tau $. Suppose $\alpha \in N\left( \mathcal{M},\tau \right)
, $ then the central valued trace $\varphi $ satisfy $\alpha (\varphi
(x))\leq \alpha (x)$, for every $x\in \mathcal{M}.$

\begin{proof}
By proposition \ref{phiproperty} (7), for any $x\in \mathcal{M}$, the
central value trace $\varphi (x)$ is in the norm closure of the convex hull
of $\{uxu^{\ast }|u\in \mathcal{U}(\mathcal{M})\}$, so there exist a net $%
\left\{ x_{\lambda }\right\} _{\lambda \in \Lambda }$ in the convex hull of $%
\{uxu^{\ast }|u\in \mathcal{U}(\mathcal{M})\}$ such that $x_{\lambda }$
converges to $\varphi \left( x\right) $. Since $\alpha $ is a continuous
norm, $\alpha \left( x_{\lambda }-\varphi \left( x\right) \right)
\rightarrow 0$, i.e., $\alpha \left( \varphi \left( x\right) \right) =%
\underset{\lambda }{\lim }\alpha \left( x_{\lambda }\right) $. Since $%
x_{\lambda }$ is in the convex hull of $\{uxu^{\ast }|u\in \mathcal{U}(%
\mathcal{M})\}$, $\alpha \left( x_{\lambda }\right) \leq \alpha \left(
x\right) $. Therefore, $\alpha (\varphi (x))\leq \alpha (x)$.
\end{proof}
\end{lemma}

\begin{theorem}
\label{cdominating}Let $\mathcal{M}$ be a finite von Neumann algebra with a
faithful, normal, tracial state $\tau $ and $\alpha $ be a normalized,
unitarily invariant, continuous norm on $(\mathcal{M},\tau )$. Then there
exists a positive $g\in L^{1}(\mathcal{Z},\tau )$ such that (i) $\rho (\cdot
)=\tau (\cdot g)$ is a faithful normal tracial state on $\mathcal{M}$, (ii) $%
\alpha $ is $c\left\Vert \cdot \right\Vert _{1,\rho }$-dominating for some $c>0$. (iii), $%
\rho (x)=\tau (xg)$ for every $x\in L^{1}(\mathcal{M},\rho ).$
\end{theorem}

\begin{proof}
Since the center $\mathcal{Z}$ of $\mathcal{M}$ is an abelian von Neumann
algebra, there is a compact subset $X$ of $\mathbb{R}$ and a regular Borel
probability measure on $X$ such that the mapping $\pi $ from $\mathcal{Z}$
to $L^{\infty }(X,\mu )$ is $\ast $-isomorphic and WOT-homeomorphic. Since $%
\alpha $ is a continuous normalized unitarily invariant norm on $(\mathcal{M}%
,\tau )$, it is easy to check $\overline{\alpha }=\alpha \circ \pi ^{-1}$
satisfying\newline
(i) $\overline{\alpha }(1)=\alpha \circ \pi ^{-1}(1)=\alpha (\pi
^{-1}(1))=\alpha (I)=1,$\newline
(ii) $\overline{\alpha }(f)=\alpha \circ \pi ^{-1}(f)=\alpha (u\pi
^{-1}(f))=\alpha (\pi ^{-1}(wf))=\alpha (\pi ^{-1}(|f|))=\overline{\alpha }%
(|f|)$, where $|f|=wf,|w|=1$ and there is a unitary $u$ such that $\pi (u)=w$%
,\newline
(iii) For given borel set $\{E_{n}\}_{n=1}^{\infty }\mathcal{\subseteq }X,$
there exist a sequence $\{e_{n}\}\subseteq \mathcal{Z}$ such that $\pi
^{-1}(\chi _{E_{n}})=e_{n}$ for every $n\in 
\mathbb{N}
.$ If $\mu (E_{n})\rightarrow 0$, then $\tau (e_{n})\rightarrow 0$. So $%
\alpha (e_{n})\rightarrow 0$ since $\alpha $ is continuous. Thus $\underset{%
n\rightarrow \infty }{\lim }\overline{\alpha }(\chi _{E_{n}})=\underset{%
n\rightarrow \infty }{\lim }\alpha \circ \pi ^{-1}(\chi _{En})=\underset{%
n\rightarrow \infty }{\lim }\alpha (e_{n})\rightarrow 0$. Thus $\overline{%
\alpha }$ is a continuous normalized gauge norm on $L^{\infty }(X,\mu )$.

By the lemma \ref{lmforcdominating}, there exists a probability measure $%
\lambda $ such that $\lambda \ll \mu $ and $\mu \ll \lambda $ and there
exists $c>0$ such that $\forall f\in L^{\infty }(X,\mu )=L^{\infty }(X,\lambda ),%
\overline{\alpha }(f)\geq c\Vert f\Vert _{1,\lambda }$. Define $\rho
_{0}(x)=\int_{X}\pi (x)d\lambda $, we check $\rho _{0}$ is a faithful normal
tracial state on $\mathcal{Z}$.\newline
(1) $\rho _{0}(I)=\int_{X}\pi (I)d\lambda =\int_{X}1d\lambda =1$,\newline
(2) $\rho _{0}(xy)=\int_{X}\pi (xy)d\lambda =\int_{X}\pi (yx)d\lambda =\rho
_{0}(yx)$,\newline
(3) Since $x_{n}\rightarrow x$ in WOT topology, $\pi (x_{n})\rightarrow \pi
(x)$ in weak* topology, i.e., $\int_{X}\pi (x_{n})d\lambda =\int_{X}\pi
(x_{n})gd\mu \rightarrow \int_{X}\pi (x)gd\mu =\int_{X}\pi (x)d\lambda .$
Thus $\rho _{0}(x_{n})\rightarrow \rho _{0}(x).$ Therefore $\rho _{0}$ is
normal.\newline
(4) For every $x\in \mathcal{Z},\rho _{0}(x^{\ast }x)=\int_{X}\pi (x^{\ast
}x)d\lambda =\int_{X}\pi (x)^{2}d\lambda =0$, so $\pi (x)^{2}=0$ and $x=0$,
which means $\rho _{0}$ is faithful.

Now claim that $\alpha $ is $c\Vert \cdot \Vert _{1,\rho }$-dominating on $(%
\mathcal{M},\rho )$. For some constant $c>0,$ $\forall x\in \mathcal{Z}$, $%
\alpha (x)=\overline{\alpha }\circ \pi (x)=\overline{\alpha }(\pi (x))\geq
c\Vert \pi (x)\Vert _{1,\lambda }=c\int_{X}|\pi (x)|d\lambda =c\int_{X}\pi
(|x|)d\lambda =c\rho (|x|)=c\Vert x\Vert _{1,\rho }.$ So we have $\alpha
(x)\geq c\Vert x\Vert _{1,\rho },\forall x\in \mathcal{Z}.$ Also, we have $%
\mathcal{M}\overset{\varphi }{\rightarrow }\mathcal{Z}\overset{\rho _{0}}{%
\rightarrow }\mathbb{C}$, where $\varphi $ is the mapping in lemma \ref%
{phiproperty}$.$ Let $\rho =\rho _{0}\circ \varphi ,$ then $\rho $ is a
state on $\mathcal{M}$, and $\forall x\in \mathcal{M},\alpha (x)\geq \alpha
(\varphi (x))\geq c\Vert \varphi (x)\Vert _{1,\rho _{0}}=c\Vert \varphi
(x)\Vert _{1,\rho }=c\Vert x\Vert _{1,\rho }$. Therefore, there exists a
faithful normal tracial state $\rho $ on $\mathcal{M}$ such that $\alpha $
is a $c\Vert \cdot \Vert _{1,\rho }$-dominating on $(\mathcal{M},\rho )$.

Since $\rho (x)=\int_{X}\pi (x)d\lambda =\int_{X}\pi (x)hd\mu ,$ where $h=%
\frac{d\lambda }{d\mu }\in L^{1}(X,\mu ),$ we can choose simple functions $%
\{h_{i}\}_{i=1}^{\infty }$ such that $0\leq h_{1}\leq h_{2}\leq \cdots $ and 
$h_{n}\rightarrow h$ as $n\rightarrow \infty .$ And also we can choose $%
0\leq x_{1}\leq x_{2}\leq \cdots $ in $\mathcal{Z}$ so that $\pi
(x_{n})=h_{n}$ for each $n.$ Therefore,%
\begin{equation*}
\rho (x)=\rho (\varphi (x))=\underset{n\rightarrow \infty }{\lim }\tau
(x_{n}\varphi (x))=\underset{n\rightarrow \infty }{\lim }\tau (\varphi
(x_{n}x))=\underset{n\rightarrow \infty }{\lim }\tau (x_{n}x))=\tau (xg),
\end{equation*}%
where $g\in L^{1}(\mathcal{Z},\tau ).$
\end{proof}

\begin{example}
Given any finite von Neumann algebra $\mathcal{M}$ with a faithful normal
tracial state $\tau $ and $\alpha \in N(\mathcal{M},\tau ),$ by theorem \ref%
{cdominating}, there exists a positive $g\in L^{1}\left( \mathcal{M},\tau
\right) $ such that $\alpha \left( x\right) \geq c\tau \left( \left\vert
x\right\vert g\right) $ for some $c>0.$ If $\Delta \left( g\right) >0,$ then 
$\alpha \in N_{\Delta }(\mathcal{M},\tau ).$
\end{example}

\section{Noncommutative Hardy spaces}

Let $\mathcal{M}$ be a finite von Neumann algebra with a faithful, normal,
tracial state $\tau .$ Given a von Neumann subalgebra $\mathcal{D}$ of $%
\mathcal{M}$, a conditional expectation $\Phi $: $\mathcal{M}\rightarrow 
\mathcal{D}$ is a positive linear map satisfying $\Phi (I)=I$ and $\Phi
(x_{1}yx_{2})=x_{1}\Phi (y)x_{2}$ for all $x_{1},x_{2}\in \mathcal{D}$ and $%
y\in \mathcal{M}.$ There exists a unique conditional expectation $\Phi
_{\tau }$: $\mathcal{M}\rightarrow \mathcal{D}$ satisfying $\tau \circ \Phi
_{\tau }(x)=\tau (x)$ for every $x\in \mathcal{M}$. Now we recall
noncommutative classical Hardy spaces $H^{\infty }(\mathbb{T})$ in \cite%
{Arveson}.

\begin{definition}
Let $\mathcal{A}$ be a weak* closed unital subalgebra of $\mathcal{M}$, and
let $\Phi _{\tau }$ be the unique faithful normal trace preserving
conditional expectation from $\mathcal{M}$ onto the diagonal von Neumann
algebra $\mathcal{D}=\mathcal{A}\cap \mathcal{A^{\ast }}$. Then $\mathcal{A}$
is called a finite, maximal subdiagonal subalgebra of $\mathcal{M}$ with
respect to $\Phi _{\tau }$ if\newline
(1) $\mathcal{A}+\mathcal{A^{\ast }}$ is weak* dense in $\mathcal{M}$,%
\newline
(2) $\Phi _{\tau }(xy)=\Phi _{\tau }(x)\Phi _{\tau }(y)$ for all $x,y\in 
\mathcal{A}$.\newline
Such $\mathcal{A}$ will be denoted by $H^{\infty }$, and $\mathcal{A}$ is
also called a noncommutative Hardy space.
\end{definition}

\begin{example}
Let $\mathcal{M}=L^{\infty }(\mathbb{T},\mu )$, and $\tau (f)=\int fd\mu $
for all $f\in L^{\infty }(\mathbb{T},\mu )$. Let $\mathcal{A}=H^{\infty }(%
\mathbb{T},\mu )$, then $\mathcal{D}=H^{\infty }(\mathbb{T},\mu )\cap
H^{\infty }(\mathbb{T},\mu )^{\ast }=\mathbb{C}$. Let $\Phi _{\tau }$ be the
mapping from $L^{\infty }(\mathbb{T},\mu )$ onto $\mathbb{C}$ defined by $%
\Phi _{\tau }(f)=\int fd\mu $. Then $H^{\infty }(\mathbb{T},\mu )$ is a
finite, maximal subdiagonal subalgebra of $L^{\infty }(\mathbb{T},\mu )$.
\end{example}

\begin{example}
Let $\mathcal{M}=\mathcal{M}_{n}(\mathbb{C})$ be with the usual trace $\tau $%
. Let $\mathcal{A}$ be the subalgebra of lower triangular matrices, now $%
\mathcal{D}$ is the diagonal matrices and $\Phi _{\tau }$ is the natural
projection onto the diagonal matrices. Then $\mathcal{A}$ is a finite maximal
subdiagonal subalgebra of $\mathcal{M}_{n}(\mathbb{C})$.
\end{example}

Let $\mathcal{M}$ be a finite von Neumann algebra with a faithful, normal,
tracial state $\tau $, $\Phi _{\tau }$ be the conditional expectation and $%
\alpha $ be a determinant, normalized, unitarily invariant, continuous norm
on $\mathcal{M}$. Let $L^{\alpha }(\mathcal{M},\tau )$ be the $\alpha $
closure of $\mathcal{M}$,i.e., $L^{\alpha }(\mathcal{M},\tau )=[\mathcal{M}%
]_{\alpha }$. Similarly, $H^{\alpha }(\mathcal{M},\tau )=[H^{\infty }(%
\mathcal{M},\tau )]_{\alpha }$, $H_{0}^{\infty }(\mathcal{M},\tau )=\ker
(\Phi _{\tau })\cap H^{\infty }(\mathcal{M},\tau )$ and $H_{0}^{\alpha }(%
\mathcal{M},\tau )=\ker (\Phi _{\tau })\cap H^{\alpha }(\mathcal{M},\tau ).$
If we take $\alpha =\Vert \cdot \Vert _{p}$, then $L^{p}(\mathcal{M},\tau )=[%
\mathcal{M}]_{p}$, $H^{p}(\mathcal{M},\tau )=[H^{\infty }(\mathcal{M},\tau
)]_{p}.$ Recall $\rho $ is a faithful normal tracial state on $\mathcal{M}$
satisfying all three conditions in theorem \ref{cdominating}. We define the
noncommutative Hardy spaces $H^{1}(\mathcal{M},\rho )$ and $H_{0}^{1}(%
\mathcal{M},\rho )$ by $H^{1}(\mathcal{M},\rho )=\overline{H^{\infty }(%
\mathcal{M},\tau )}^{\Vert \cdot \Vert _{1,\rho }}$ and $H_{0}^{1}(\mathcal{M%
},\rho )=\overline{H_{0}^{\infty }(\mathcal{M},\tau )}^{\Vert \cdot \Vert
_{1,\rho }}$. In \cite{Saito}, K. S. Satio characterized the noncommutative
Hardy spaces $H^{p}(\mathcal{M},\tau )$ and $H_{0}^{p}(\mathcal{M},\tau ).$
Recall $H^{p}\left( \mathcal{M},\tau \right) =\left\{ x\in L^{p}(\mathcal{M}%
,\tau )\text{, }\tau \left( xy\right) =0\text{, for all }y\in H_{0}^{\infty
}\right\} $ for $1\leq p<\infty $, also we have $H_{0}^{p}\left( \mathcal{M}%
,\tau \right) =\left\{ x\in L^{p}(\mathcal{M},\tau )\text{, }\tau \left(
xy\right) =0\text{, }\forall y\in H^{\infty }\right\} $. In this paper, we
get similar result for noncommutative Hardy spaces $H^{p}(\mathcal{M},\rho )$
and $H_{0}^{p}(\mathcal{M},\rho )$ by using the inner-outer factorization
and the properties of outer functions in noncommutative Hardy spaces from
papers \cite{B.L.} and \cite{B.L.2}.

\begin{lemma}
\label{phifactorization}\text{(from }\cite{B.L.}\text{)} If $H^{\infty }$ is
a maximal subdiagonal algebra, then $x\in H^{p}(\mathcal{M},\tau )$ with $%
\Delta (x)>0$ iff $x=uy$ for an inner $u$ and a strongly outer $y\in H^{p}(%
\mathcal{M},\tau )$ for $1\leq p\leq \infty $. The factorization is unique
up to a unitary in $\mathcal{D}$.
\end{lemma}

\begin{lemma}
\text{(from }\cite{B.L.2}\text{) }Let $\Phi _{\tau }$ be the conditional
expectation on $\mathcal{M}$. Then $x\in H^{p}(\mathcal{M},\tau )$ is outer
if and only if $\Phi _{\tau }(x)$ is outer in $L^{p}(\mathcal{D})$ and $%
\overline{xH_{0}^{\infty }(\mathcal{M},\tau )}^{\Vert \cdot \Vert _{p,\tau
}}=H_{0}^{p}(\mathcal{M},\tau )$ for $1\leq p\leq \infty $.
\end{lemma}

\begin{lemma}
\label{outerh}If $\alpha \in N_{\Delta }(\mathcal{M},\tau ),$ then there
exits a faithful tracial state $\rho $ and a strongly outer $h$ in $H^{1}(%
\mathcal{M},\tau )$ such that $g=|h|$ and $hH^{1}(\mathcal{M},\rho )=H^{1}(%
\mathcal{M},\tau )$.
\end{lemma}

\begin{proof}
Since $\alpha \in N_{\Delta }(\mathcal{M},\tau ),$ $\Delta (g)>0$. By lemma %
\ref{phifactorization}, $g=|h|$ for a strongly outer $h\in H^{1}(\mathcal{M}%
,\tau )$. Let $\rho \left( x\right) =\tau \left( xg\right) $, $\forall x\in 
\mathcal{M}$, by theorem \ref{cdominating}, $\rho $ is a faithful normal
tracial state on $\mathcal{M}$. Then we define $U:L^{1}(\mathcal{M},\rho
)\longrightarrow L^{1}(\mathcal{M},\tau )$ by $Ux=hx$, which is a surjective
isometry:%
\begin{equation*}
\left\Vert U(x)\right\Vert _{1,\tau }=\left\Vert xg\right\Vert _{1,\tau
}=\tau (\left\vert xg\right\vert )=\tau (\left\vert x\right\vert g)=\rho
(\left\vert x\right\vert )=\left\Vert x\right\Vert _{1,\rho }.
\end{equation*}%
Since $g\in gH^{1}(\mathcal{M},\rho )$ and $H^{1}(\mathcal{M},\tau
)\subseteq H^{1}(\mathcal{M},\rho )$, $gH^{\infty }(\mathcal{M},\tau
)\subseteq gH^{1}(\mathcal{M},\rho )$, i.e. $vhH^{\infty }(\mathcal{M},\tau
)\subseteq gH^{1}(\mathcal{M},\rho )=vhH^{1}(\mathcal{M},\rho )=hH^{1}(%
\mathcal{M},\rho ).$ Since $h$ is a strongly outer in $H^{1}(\mathcal{M}%
,\tau )$, $hH^{1}(\mathcal{M},\rho )=H^{1}(\mathcal{M},\tau ).$
\end{proof}

\begin{corollary}
\label{H(0,1)}Let $\Phi _{\tau }$ be the conditional expectation on $%
\mathcal{M}$. If $\alpha \in N_{\Delta }(\mathcal{M},\tau ),$ then there
exists a faithful normal tracial state $\rho $ such that\newline
(1) $H^{1}(\mathcal{M},\rho )=\{x\in L^{1}(\mathcal{M},\rho ):\rho (xy)=0$
for all $y\in H_{0}^{\infty }\}$,\newline
(2) $H_{0}^{1}(\mathcal{M},\rho )=\{x\in L^{1}(\mathcal{M},\rho ):\rho (xy)=0
$ for all $y\in H^{\infty }\}$,\newline
(3) $H_{0}^{1}(\mathcal{M},\rho )=\{x\in H^{1}(\mathcal{M},\rho ):\Phi
_{\tau }(xh)=0\}$.
\end{corollary}

\begin{proof}
Since $\alpha \in N_{\Delta }(\mathcal{M},\tau ),$ there exists a positive $%
g\in L^{1}\left( \mathcal{M},\tau \right) $ and $\Delta \left( g\right) >0$
such that $\alpha \left( x\right) \geq c\tau \left( \left\vert x\right\vert
g\right) $ for some $c>0.$ We define $\rho \left( x\right) =\tau \left(
xg\right) ,\forall x\in \mathcal{M}$, $\rho $ is a faithful normal tracial
state on $\mathcal{M}.$ By lemma \ref{outerh} and $H^{1}\left( \mathcal{M}%
,\tau \right) =\left\{ x\in L^{1}(\mathcal{M},\tau ),\tau \left( xy\right) =0%
\text{ for all }y\in H_{0}^{\infty }\right\}$, we have (1). For (2), We
know $\overline{H_{0}^{\infty }(\mathcal{M},\tau )}^{\Vert \cdot \Vert
_{1,\rho }}=H_{0}^{1}(\mathcal{M},\rho )$, and $hH_{0}^{1}(\mathcal{M},\rho )=h%
\overline{H_{0}^{\infty }(\mathcal{M},\tau )}^{\Vert \cdot \Vert _{1,\rho }}=%
\overline{hH_{0}^{\infty }(\mathcal{M},\tau )}^{\Vert \cdot \Vert _{1,\tau
}}=H_{0}^{1}(\mathcal{M},\tau )$ since $h$ is outer in $H^{1}(\mathcal{M}%
,\tau )$. The last statement is clearly by \cite{Saito}.
\end{proof}

\begin{proposition}
If $\alpha \in N_{\Delta }(\mathcal{M},\tau ),$ then there exists a faithful
normal tracial state $\rho $ such that%
\begin{equation*}
H^{\alpha }(\mathcal{M},\rho )=\{x\in L^{\alpha }(\mathcal{M},\rho ):\rho
(xy)=0\text{ for all }y\in H_{0}^{1}(\mathcal{M},\rho )\cap (L^{\alpha }(%
\mathcal{M},\rho ))^{\#}\},
\end{equation*}%
where $(L^{\alpha }(\mathcal{M},\rho ))^{\#}$ is the dual space of $%
L^{\alpha }(\mathcal{M},\rho ).$
\end{proposition}

\begin{proof}
Since $\alpha \in N_{\Delta }(\mathcal{M},\tau ),$ then there exists a
faithful normal tracial state $\rho $ on $\mathcal{M}$ such that $\rho
\left( x\right) =\tau \left( xg\right) $ for some positive $g\in L^{1}\left( 
\mathcal{Z},\tau \right) $ and the determinant of $g$ is positive, which
means $\alpha \in N_{1}(\mathcal{M},\rho )$. Let $\mathcal{J}=\{x\in
L^{\alpha }(\mathcal{M},\rho ):\rho (xy)=0$ for all $y\in H_{0}^{1}(\mathcal{%
M},\rho )\cap (L^{\alpha }(\mathcal{M},\rho ))^{\#}\}$. Suppose $x\in
H^{\infty }$. If $y\in H_{0}^{1}(\mathcal{M},\rho )\cap (L^{\alpha }(%
\mathcal{M},\rho ))^{\#}\subseteq H_{0}^{1}(\mathcal{M},\rho )$, then it
follows from corollary \ref{H(0,1)} that $\rho (xy)=0$, for all $x\in \mathcal{%
J}$, and so $H^{\infty }\subseteq \mathcal{J}$. We claim that $\mathcal{J}$
is $\alpha $-closed in $L^{\alpha }(\mathcal{M},\rho )$. In fact, suppose $%
\{x_{n}\}$ is a sequence in $\mathcal{J}$ and $x\in L^{\alpha }(\mathcal{M}%
,\rho )$ such that $\alpha (x_{n}-x)\rightarrow 0$. If $y\in H_{0}^{1}(%
\mathcal{M},\rho )\cap (L^{\alpha }(\mathcal{M},\rho ))^{\#}$, then by the
generalized Holder's inequality, we have%
\begin{equation*}
|\rho (xy)-\rho (x_{n}y)|=|\rho ((x-x_{n})y)|\leq \alpha (x-x_{n})\alpha
^{\prime }\rightarrow 0.
\end{equation*}%
Which follows that $\rho (xy)=\underset{n\rightarrow \infty }{\lim }\rho
(x_{n}y)=0$ for all $y\in H_{0}^{1}(\mathcal{M},\rho )\cap (L^{\alpha }(%
\mathcal{M},\rho ))^{\#}$. By the definition of $\mathcal{J}$, we know $x\in 
\mathcal{J}$. Hence $\mathcal{J}$ is closed in $L^{\alpha }(\mathcal{M},\rho
)$. Therefore, $H^{\alpha }(\mathcal{M},\rho )=[H^{\infty }]_{\alpha
}\subseteq \mathcal{J}$.\newline
Next, we show that $H^{\alpha }(\mathcal{M},\rho )=\mathcal{J}$. Assume, via
contradiction, that $H^{\alpha }(\mathcal{M},\rho )\subsetneq \mathcal{J}%
\subseteq L^{\alpha }(\mathcal{M},\rho )$. By the Hahn-Banach theorem, there
is a linear functional $\phi \in (L^{\alpha }(\mathcal{M},\rho ))^{\#}$ and $%
x\in \mathcal{J}$ such that\newline
(a) $\phi (x)\neq 0$,\newline
(b) $\phi (y)=0$ for all $y\in H^{\alpha }(\mathcal{M},\rho )$.\newline
In the beginning of this proof, we know $\alpha \in N_{1}(\mathcal{M},\rho )$%
, which means $\alpha $ is normalized, unitarily invariant $\Vert \cdot
\Vert _{1}$-dominating, continuous norm on $(\mathcal{M},\rho )$, it follows
from \cite{Chen} that there exists a $\xi \in (L^{\alpha }(\mathcal{M},\rho
))^{\#}$ such that\newline
(c) $\phi (z)=\rho (z\xi )$ for all $z\in L^{\alpha }(\mathcal{M},\rho )$.%
\newline
Hence from (b) and (c) we can conclude that\newline
(d) $\rho (y\xi )=\phi (y)=0$ for every $y\in H^{\infty }\subseteq H^{\alpha
}(\mathcal{M},\rho )\subseteq L^{\alpha }(\mathcal{M},\rho )$.\newline
Since $\phi \in (L^{\alpha }(\mathcal{M},\rho ))^{\#}\subseteq L^{1}(%
\mathcal{M},\rho )$, so $\xi \in H_{0}^{1}(\mathcal{M},\rho )$, which means $%
\xi \in H_{0}^{1}(\mathcal{M},\rho )\cap (L^{\alpha }(\mathcal{M},\rho
))^{\#}$. Combining with the fact that $x\in \mathcal{J}=\{x\in L^{\alpha }(%
\mathcal{M},\rho ):\rho (xy)=0$, $\forall y\in H_{0}^{1}(\mathcal{M},\rho
)\cap (L^{\alpha }(\mathcal{M},\rho ))^{\#}\}$, we obtain that $\rho (x\xi
)=0$. Note, again, that $x\in \mathcal{J}\subseteq L^{\alpha }(\mathcal{M}%
,\rho )$. From (a) and (c), it follows that $\rho (x\xi )=\phi (x)\neq 0$.
This is a contradiction. Therefore\newline
$H^{\alpha }(\mathcal{M},\rho )=\{x\in L^{\alpha }(\mathcal{M},\rho ):\rho
(xy=0)$ for all $y\in H_{0}^{1}(\mathcal{M},\rho )\cap (L^{\alpha }(\mathcal{%
M},\rho ))^{\#}\}$.
\end{proof}

\begin{lemma}
\label{conexpt}\text{(from }\cite{Bekjan}\text{) }The conditional
expectation $\Phi _{\tau }$ is multiplicative on Hardy spaces. More
precisely, $\Phi _{\tau }(xy)=\Phi _{\tau }(x)\Phi _{\tau }(y)$ for $x\in
H^{p}(\mathcal{M},\tau )$, $y\in H^{q}(\mathcal{M},\tau )$ and $xy\in H^{r}(%
\mathcal{M},\tau )$ with $0<p,q,r<\infty $ and $\frac{1}{p}+\frac{1}{q}=%
\frac{1}{r}$.
\end{lemma}

\begin{theorem}
\label{Halpha}If $\alpha \in N_{\Delta }(\mathcal{M},\tau ),$ then there
exists a faithful normal tracial state $\rho $ such that $H^{\alpha }(%
\mathcal{M},\rho )=H^{1}(\mathcal{M},\rho )\cap L^{\alpha }(\mathcal{M},\rho
)$.
\end{theorem}

\begin{proof}
Since $\alpha \in N_{\Delta }(\mathcal{M},\tau ),$ there exists a positive $%
g\in L^{1}\left( \mathcal{M},\tau \right) $ and $\Delta \left( g\right) >0$
such that $\alpha \left( x\right) \geq c\tau \left( \left\vert x\right\vert
\cdot g\right) $ for some $c>0.$ We define a faithful normal tracial state $%
\rho \left( x\right) =\tau \left( xg\right) $, $\forall x\in \mathcal{M}$.
Since $\alpha \left( x\right) \geq c\tau \left( \left\vert x\right\vert
g\right) =c\rho \left( \left\vert x\right\vert \right) =c\left\Vert
x\right\Vert _{1,\rho },$ $\alpha $ is $\Vert \cdot \Vert _{1,\rho }$%
-dominating, so $\alpha $-convergence implies $\Vert \cdot \Vert _{1,\rho }$%
-convergence, thus $H^{\alpha }(\mathcal{M},\rho )=\overline{H^{\infty }(%
\mathcal{M},\rho )}^{\alpha }=\overline{H^{\infty }(\mathcal{M},\rho )}%
^{\Vert \cdot \Vert _{1,\rho }}=H^{1}(\mathcal{M},\rho )$. Also, $H^{\alpha
}(\mathcal{M},\rho )=\overline{H^{\infty }(\mathcal{M},\rho )}^{\alpha
}\subseteq L^{\alpha }(\mathcal{M},\rho )$. Therefore, $H^{\alpha }(\mathcal{%
M},\rho )\subseteq H^{1}(\mathcal{M},\rho )\cap L^{\alpha }(\mathcal{M},\rho
)$.

To prove $H^{1}(\mathcal{M},\rho )\cap L^{\alpha }(\mathcal{M},\rho
)\subseteq H^{\alpha }(\mathcal{M},\rho )$. Suppose $x\in H^{1}(\mathcal{M}%
,\rho )\cap L^{\alpha }(\mathcal{M},\rho )$, then $x\in L^{\alpha }(\mathcal{%
M},\rho )$. Assume that $y\in H_{0}^{1}(\mathcal{M},\rho )\cap (L^{\alpha }(%
\mathcal{M},\rho ))^{\#}$. So $\Phi _{\tau }(hy)=0$. Note that $hx\in H^{1}(%
\mathcal{M},\tau )$, $hy\in H_{0}^{1}(\mathcal{M},\tau )$ and $hxhy\in H^{1}(%
\mathcal{M},\tau )H_{0}^{1}(\mathcal{M},\tau )\subseteq H^{\frac{1}{2}}$.
From theorem 2.1 in \cite{Bekjan}, and lemma \ref{conexpt} we know that $%
\Phi _{\tau }(hxhy)\in L^{\frac{1}{2}}(\mathcal{D},\tau )$ and $\Phi _{\tau
}(hxhy)=\Phi _{\tau }(hx)\Phi _{\tau }(hy)=0$. Moreover, $x\in L^{\alpha }(%
\mathcal{M},\rho )$ and $y\in (L^{\alpha }(\mathcal{M},\rho ))^{\#}$, from 
\cite{Chen}, $xy\in L^{\alpha }(\mathcal{M},\rho )\subseteq L^{1}(\mathcal{M}%
,\rho )$. So $hxy\in L^{1}(\mathcal{M},\tau )$, and $\Phi _{\tau }(hxy)$ is
also in $L^{1}(\mathcal{M},\tau )$. Thus $\rho (xy)=\tau (hxy)=\tau (\Phi
_{\tau }(hxy))=\tau (0)=0$.

Now we check $\Phi _{\tau }(hxy)=0$. Since $h$ is strongly outer in $H^{1}(%
\mathcal{M},\rho )$, there is a sequence $\{a_{n}\}$ in $H^{\infty }$ such
that $a_{n}h\rightarrow 1$ in $\Vert \cdot \Vert _{1}$ norm. Therefore, $%
\Vert hxyha_{n}-hxy\Vert _{\frac{1}{2}}=\Vert hxy(ha_{n}-1)\Vert _{\frac{1}{2%
}}\leq \Vert hxy\Vert _{1}\Vert ha_{n}-1\Vert _{1}\rightarrow 0$ as $%
n\rightarrow \infty $. And by theorem 2.1 in \cite{Bekjan}, $\Phi _{\tau
}(hxyha_{n})\rightarrow \Phi _{\tau }(hxy)$. Also, we have $\Phi _{\tau
}(hxyha_{n})=\Phi _{\tau }(hx)\Phi _{\tau }(hy)\Phi _{\tau }(a_{n})=0$, so $%
\Phi _{\tau }(hxy)=0$. By the definition of $\mathcal{J}$, we conclude that $%
x\in \mathcal{J}$. Therefore $H^{1}(\mathcal{M},\rho )\cap L^{\alpha }(%
\mathcal{M},\rho )\subseteq \mathcal{J}=H^{\alpha }(\mathcal{M},\rho )$.
\end{proof}

\section{Beurling's invariant subspace theorem}

In this section, we extend the Chen-Hadwin-Shen theorem for continuous
normalized unitarily invariant norms on $(\mathcal{M},\tau).$

First, we will prove the factorization theorem, in order to do this, we need
the following lemma.

\begin{lemma}
\label{delta}\text{(from }\cite{H.S.}\text{) }Let $x\in L^{p}(\mathcal{M})$, 
$p>0$, then we have\newline
(1) $\Delta (x)=\Delta (x^{\ast })=\Delta (|x|)$,\newline
(2) $\Delta (xy)=\Delta (x)\Delta (y)=\Delta (yx)$ for any $y\in L^{s}(%
\mathcal{M})$, $s>0$.
\end{lemma}

\begin{theorem}
\label{factorization}Suppose $\alpha \in N_{\Delta }\left( \mathcal{M},\tau
\right) $, there exists a faithful normal tracial state $\rho $ on $\mathcal{%
M}$ such that $\rho \left( x\right) =\tau \left( xg\right) $ for some
positive $g\in L^{1}\left( \mathcal{Z},\tau \right) $ and the determinant of 
$g$ is positive. If $x\in \mathcal{M}$ and $x^{-1}\in L^{\alpha }\left( 
\mathcal{M},\rho \right) ,$ then there are unitary operators $u_{1},u_{2}\in 
\mathcal{M}$ and $s_{1},s_{2}\in H^{\infty }$ such that $%
x=u_{1}s_{1}=s_{2}u_{2}$ and $s_{1}^{-1},s_{2}^{-1}\in H^{\alpha }(\mathcal{M%
},\rho )$.
\end{theorem}

\begin{proof}
Since $\alpha \in N_{\Delta }\left( \mathcal{M},\tau \right) $, the first
statement is clear. Suppose $k\in \mathcal{M}$ with $k^{-1}\in L^{\alpha }(%
\mathcal{M},\rho ).$ Assume that $k=v\left\vert k\right\vert $ is the polar
decomposition of $k$ in $\mathcal{M}$, where $v$ is a unitary in $\mathcal{M}
$ and $\left\vert k\right\vert \in \mathcal{M}.$ Since $log(|k|)\leq
|k|,log(|h|)-log(|k|)=log(|h||k|^{-1})\leq |h||k|$ and $-\log (|k|)\leq
|h||k|-log(|h|),$ $|log(|k|)|\leq |k|+(|h||k|-log(|h|))$, so $\Delta
(\left\vert k\right\vert )=e^{\tau (\log \left\vert k\right\vert )}>0$ and $%
\left\vert k\right\vert \in L^{1}(\mathcal{M})^{+}$. By corollary 4.17 in 
\cite{B.L.}, there exists a strongly outer $s\in H^{1}(\mathcal{M},\tau )$
and $s=u_{1}\left\vert s\right\vert $ is the polar decomposition of $s$ such
that $\left\vert k\right\vert =\left\vert s\right\vert .$ Since $\left\vert
k\right\vert \in \mathcal{M}$, $\left\vert s\right\vert \in \mathcal{M},$
therefore, $s\in \mathcal{M}$ and $s\in H^{1}(\mathcal{M},\tau )$ implies $%
s\in H^{\infty }(\mathcal{M},\tau ).$ Also, we have $\left\vert k\right\vert
=u_{1}^{\ast }s,$ so $k=vu_{1}^{\ast }s=us,$ where $u=vu_{1}^{\ast }.$

Now we check $s^{-1}\in H^{\alpha }(\mathcal{M},\rho ).$ First, $k^{-1}\in
L^{\alpha }(\mathcal{M},\rho )\subseteq L^{1}(\mathcal{M},\rho ),$ so $%
hk^{-1}\in L^{1}(\mathcal{M},\tau ).$ Since $k^{-1}=\left\vert k\right\vert
^{-1}v^{\ast }\in L^{\alpha }(\mathcal{M},\rho ),$ $\left\vert k\right\vert
^{-1}\in L^{\alpha }(\mathcal{M},\rho )\subseteq L^{1}(\mathcal{M},\rho )$
and $\left\vert h\right\vert \left\vert k\right\vert ^{-1}\in L^{1}(\mathcal{%
M},\tau )$. \newline
$\Delta (\left\vert h\right\vert \left\vert k\right\vert ^{-1})=\Delta
(\left\vert h\right\vert )\Delta (\left\vert k\right\vert ^{-1})>0$ by lemma %
\ref{delta}. Then there exists a strongly outer $f\in H^{1}(\mathcal{M},\tau
)$ such that $\left\vert h\right\vert \left\vert k\right\vert
^{-1}=\left\vert f\right\vert .$ Since $H^{1}(\mathcal{M},\tau )=hH^{1}(%
\mathcal{M},\rho ),$ there exists $f_{1}\in H^{1}(\mathcal{M},\rho )$ such
that $f=hf_{1}.$ Since $\Delta (fs)=\Delta (f)\Delta (s)>0$, by lemma \ref%
{phifactorization}, $fs$ is outer. And $\left\vert f\right\vert \left\vert
s\right\vert =\left\vert h\right\vert \left\vert k\right\vert
^{-1}\left\vert s\right\vert $, so $\left\vert h\right\vert =\left\vert
f\right\vert \left\vert s\right\vert $. Therefore, $\left\vert h\right\vert
=u_{2}^{\ast }fu_{1}^{\ast }s,$ i.e., $h=u_{3}^{\ast }u_{2}^{\ast
}fu_{1}^{\ast }s$, $hf_{1}=f=u_{2}u_{3}hs^{-1}u_{1}$, so $%
s^{-1}=h^{-1}u_{3}^{\ast }u_{2}^{\ast }hf_{1}u_{1}^{\ast }=u_{3}^{\ast
}u_{2}^{\ast }f_{1}u_{1}^{\ast }\in H^{1}(\mathcal{M},\rho )$. Also, we know 
$s^{-1}\in L^{\alpha }(\mathcal{M},\rho ).$ Therefore, by theorem \ref%
{Halpha}, $s^{-1}\in L^{\alpha }(\mathcal{M},\rho )\cap H^{1}(\mathcal{M}%
,\rho )=H^{\alpha }(\mathcal{M},\rho ).$
\end{proof}

The following density theorem also plays an important role in the proof of
our main result of the paper.

\begin{theorem}
\label{density}Let $\alpha \in N_{\Delta }\left( \mathcal{M},\tau \right) $,
then there exists a faithful normal tracial state $\rho $ on $\mathcal{M}$
such that $\rho \left( x\right) =\tau \left( xg\right) $ for some positive $%
g\in L^{1}\left( \mathcal{Z},\tau \right) $ and the determinant of $g$ is
positive. Also, if $\mathcal{W}$ is a closed subspace of $L^{\alpha }(%
\mathcal{M},\rho )$ and $\mathcal{N}$ is a weak* closed linear subspace of $%
\mathcal{M}$ such that $\mathcal{W}H^{\infty }\subset \mathcal{W}$ and $%
\mathcal{N}H^{\infty }\subset \mathcal{N}$, then\newline
(1) $\mathcal{N}=[\mathcal{N}]_{\alpha }\cap \mathcal{M}$,\newline
(2) $\mathcal{W}\cap \mathcal{M}$ is weak* closed in $\mathcal{M}$,\newline
(3) $\mathcal{W}=[\mathcal{W}\cap \mathcal{M}]_{\alpha }$,\newline
(4) If $\mathcal{S}$ is a subspace of $\mathcal{M}$ such that $\mathcal{S}%
H^{\infty }\subset \mathcal{S}$, then $[\mathcal{S}]_{\alpha }=[\overline{%
\mathcal{S}}^{w\ast }]_{\alpha }$, where $\overline{\mathcal{S}}^{w\ast }$
is the weak*-closure of $\mathcal{S}$ in $\mathcal{M}$.
\end{theorem}

\begin{proof}
Since $\alpha \in N_{\Delta }\left( \mathcal{M},\tau \right) $, clearly,
there exists a faithful normal tracial state $\rho $ on $\mathcal{M}$ by theorem \ref{cdominating}. For
(1), it is clear that $\mathcal{N}\subseteq \lbrack \mathcal{N}]_{\alpha
}\cap \mathcal{M}$. Assume, via contradiction, that $\mathcal{N}\subsetneq
\lbrack \mathcal{N}]_{\alpha }\cap \mathcal{M}$. Note that $\mathcal{N}$ is
a weak* closed linear subspace of $\mathcal{M}$ and $L^{1}(\mathcal{M},\rho )
$ is the predual space of $(\mathcal{M},\rho )$. It follows from the
Hahn-Banach theorem that there exist a $\xi \in L^{1}(\mathcal{M},\rho )$
and a $x\in \lbrack \mathcal{N}]_{\alpha }\cap \mathcal{M}$ such that\newline
(a) $\rho (\xi x)\neq 0$ and
(b) $\rho (\xi y)=0$ for all $y\in \mathcal{N}$.\newline
We claim that there exists a $z\in \mathcal{M}$ such that\newline
(a$^{\prime }$) $\rho (zx)\neq 0$ and
(b$^{\prime }$) $\rho (zy)=0$ for all $y\in \mathcal{N}$. Actually assume
that $\xi =|\xi ^{\ast }|v$ is the polar decomposition of $\xi \in L^{1}(%
\mathcal{M},\rho )$, where $v$ is a unitary element in $\mathcal{M}$ and $%
|\xi ^{\ast }|$ is in $L^{1}(\mathcal{M},\rho )$ is positive. Let $f$ be a
function on $[0,\infty )$ defined by the formula $f(t)=1$ for $0\leq t\leq 1$
and $f(t)=1/t$ for $t>1$. We define $k=f(|\xi ^{\ast }|)$ by the functional
calculus. Then by the construction of $f$, we know that $k\in \mathcal{M}$
and $k^{-1}=f^{-1}(|\xi ^{\ast }|)\in L^{1}(\mathcal{M},\rho )$. It follows
from theorem \ref{factorization} that there exist a unitary operator $u\in 
\mathcal{M}$ and $s\in H^{\infty }$ such that $k=us$ and $s^{-1}\in H^{1}(%
\mathcal{M},\rho )$. Therefore, we can further assume that $\{{t_{n}\}}%
_{n=1}^{\infty }$ is a sequence of elements in $H^{\infty }$ such that $%
\Vert s^{-1}-t_{n}\Vert _{1,\rho }$. Observe that\newline
(i) Since $s,t_{n}$ are in $H^{\infty }$, for each $y\in \mathcal{N}$ we
have that $yt_{n}s\in \mathcal{N}H^{\infty }\subseteq \mathcal{N}$ and $\rho
(t_{n}s\xi y)=\rho (\xi yt_{n}s)=0$,\newline
(ii) We have $s\xi =(u\ast u)s(|\xi ^{\ast }|v)=u\ast (k|\xi ^{\ast }|)v\in 
\mathcal{M}$, by the definition of $k$,\newline
(iii) From (a) and (i), we have $0\neq \rho (\xi x)=\rho (s^{-1}s\xi x)=%
\underset{n\rightarrow \infty }{\lim }\rho (t_{n}s\xi x)$.\newline
Combining (i), (ii) and (iii), we are able to find an $N\in \mathbb{Z}$ such
that $z=t_{N}s\xi \in \mathcal{M}$ satisfying\newline
(a$^{\prime }$) $\rho (zx)\neq 0$ and
(b$^{\prime }$) $\rho (zy)=0$ for all $y\in \mathcal{N}$.\newline
Recall that $x\in \lbrack \overline{\mathcal{N}}]_{\alpha }$. Then there is
a sequence $\{{x_{n}\}}\subseteq \mathcal{N}$ such that $\alpha
(x-x_{n})\rightarrow 0$. We have\newline
$|\rho (zx_{n})-\rho (zx)|=|\rho (x-x_{n})|\leq \Vert x-x_{n}\Vert _{1,\rho
}\Vert z\Vert \rightarrow 0$.\newline
Combining with (b$^{\prime }$) we conclude that $\rho (zx)=\underset{%
n\rightarrow \infty }{\lim }\rho (zx_{n})=0$. This contradicts with the
result (a$^{\prime }$). Therefore, $\mathcal{N}=[\mathcal{N}]_{\alpha }\cap 
\mathcal{M}$.

For (2), let $\overline{\mathcal{W}\cap \mathcal{M}}^{w\ast }$ be the
weak*-closure of $\mathcal{W}\cap \mathcal{M}$ in $\mathcal{M}$. In order to
show that $\mathcal{W}\cap \mathcal{M}=\overline{\mathcal{W}\cap \mathcal{M}}%
^{w\ast }$, it suffices to show that $\overline{\mathcal{W}\cap \mathcal{M}}%
^{w\ast }\subseteq \mathcal{W}$. Assume, to the contrary, that $\overline{%
\mathcal{W}\cap \mathcal{M}}^{w\ast }\nsubseteq \mathcal{W}$. Thus there
exists an element $x$ in $\overline{\mathcal{W}\cap \mathcal{M}}^{w\ast
}\subset \mathcal{M}\subseteq L^{\alpha }(\mathcal{M},\rho )$, but $x\in 
\mathcal{W}$. Since $\mathcal{W}$ is a closed subspace of $L^{\alpha }(%
\mathcal{M},\rho )$, by the Hahn-Banach theorem, there exists a $\xi \in
L^{1}(\mathcal{M},\rho )$ such that $\rho (\xi x)\neq 0$ and $\rho (\xi y)=0$
for all $y\in \mathcal{W}$. Since $\xi \in L^{1}(\mathcal{M},\rho )$, the
linear mapping $\rho _{\xi }:\mathcal{M}\rightarrow \mathcal{%
\mathbb{C}
}$, defined by $\rho _{\xi }(a)=\rho (\xi a)$ for all $a\in \mathcal{M}$ is
weak*-continuous. Note that $x\in \overline{\mathcal{W}\cap \mathcal{M}}%
^{w\ast }$ and $\rho (\xi y)=0$ for all $y\in \mathcal{W}$. We know that $%
\rho (\xi x)=0$, which contradicts with the assumption that $\rho (\xi
x)\neq 0$. Hence $\overline{\mathcal{W}\cap \mathcal{M}}^{w\ast }\subseteq 
\mathcal{W}$, so $\mathcal{W}\cap \mathcal{M}=\overline{\mathcal{W}\cap 
\mathcal{M}}^{w\ast }$.

For (3), since $\mathcal{W}$ is $\alpha $-closed, it is easy to see $[%
\mathcal{W}\cap \mathcal{M}]_{\alpha }\subseteq \mathcal{W}$. Now we assume $%
[\mathcal{W}\cap \mathcal{M}]_{\alpha }\subsetneq \mathcal{M}\subseteq
L^{\alpha }(\mathcal{M},\rho )$. By the Hahn-Banach theorem, there exists an $%
x\in \mathcal{W}$ and $\xi \in L^{1}(\mathcal{M},\rho )$ such that $\rho
(\xi x)\neq 0$ and $\rho (\xi y)=0$ for all $y\in \lbrack \mathcal{W}\cap 
\mathcal{M}]_{\alpha }$. Let $x=v|x|$ be the polar decomposition of $x$ in $%
L^{\alpha }(\mathcal{M},\rho )$, where $v$ is a unitary element in $\mathcal{%
M}$. Let $f$ be a function on $[0,\infty )$ defined by the formula $f(t)=1$
for $0\leq t\leq 1$ and $f(t)=1/t$ for $t>1$. We define $k=f(|x|)$ by the
functional calculus. Then by the construction of $f$, we know that $k\in 
\mathcal{M}$ and $k^{-1}=f^{-1}(|x|)\in L^{\alpha }(\mathcal{M},\rho )$. It
follows from theorem \ref{factorization} that there exist a unitary operator 
$u\in \mathcal{M}$ and $s\in H^{\infty }$ such that $k=su$ and $s^{-1}\in
H^{\alpha }(\mathcal{M},\rho )$. A little computation shows that $|x|k\in 
\mathcal{M}$ which implies that $xs=xsuu^{\ast }=xku^{\ast }=v(|x|k)u^{\ast
}\in \mathcal{M}$. Since $s\in H^{\infty }$, we know $xs\in \mathcal{W}%
H^{\infty }\subseteq \mathcal{W}$ and thus $xs\in \mathcal{W}\cap \mathcal{M}
$. Furthermore, note that $\mathcal{W}\cap \mathcal{M}H^{\infty }\subseteq 
\mathcal{W}\cap \mathcal{M}$. Thus, if $t\in H^{\infty }$ we see $xst\in 
\mathcal{W}\cap \mathcal{M}$, and $\rho (\xi xst)=0$. Since $H^{\infty }$ is
dense in $H^{\alpha }(\mathcal{M},\rho )$ and $\xi \in L^{1}(\mathcal{M}%
,\rho )$, it follows that $\rho (\xi xst)=0$ for all $t\in H^{\alpha }(%
\mathcal{M},\rho )$. Since $s^{-1}\in H^{\alpha }(\mathcal{M},\rho )$, we
see that $\rho (\xi x)=\rho (\xi xss^{-1})=0$. This contradicts with the
assumption that $\rho (\xi x)\neq 0$ . Therefore $\mathcal{W}=[\mathcal{W}%
\cap \mathcal{M}]_{\alpha }$.

For (4), assume that $\mathcal{S}$ is a subspace of $\mathcal{M}$ such that $%
\mathcal{S}H^{\infty }\subset \mathcal{S}$ and $\overline{\mathcal{S}}%
^{w\ast }$ is weak*-closure of $\mathcal{S}$ in $\mathcal{M}$. Then $[%
\mathcal{S}]_{\alpha }H^{\infty }\subseteq \lbrack \mathcal{S}]_{\alpha }$.
Note that $\mathcal{S}\subseteq \lbrack \mathcal{S}]_{\alpha }\cap \mathcal{M%
}$. From (2), we know that $[\mathcal{S}]_{\alpha }\cap \mathcal{M}$ is
weak*-closed. Therefore, $\overline{\mathcal{S}}^{w\ast }\subseteq \lbrack 
\mathcal{S}]_{\alpha }\cap \mathcal{M}$. Hence $[\overline{\mathcal{S}}%
^{w\ast }]_{\alpha }\subseteq \lbrack \mathcal{S}]_{\alpha }$ and $[\mathcal{%
S}]_{\alpha }=[\overline{\mathcal{S}}^{w\ast }]_{\alpha }$.
\end{proof}
Before we obtain our main result in the paper, we call the definitions of internal column sum of a family of subspaces, and
the lemma in \cite{B.L.}.

\begin{definition}
Let $X$ be a closed subspace of $L^{\alpha }(\mathcal{M},\tau )$ with $%
\alpha \in N_{\Delta }\left( \mathcal{M},\tau \right) $. Then $X$ is called
an internal column sum of a family of closed subspaces $\{X_{\lambda
}\}_{\lambda \in \Lambda }$ of $L^{\alpha }(\mathcal{M},\tau )$, denoted by 
$X=\bigoplus_{\lambda \in \Lambda }^{col}X_{\lambda }$ if\newline
(1) $X_{\mu }^{\ast }X_{\lambda }=\{0\}$ for all distinct $\lambda ,\mu \in
\Lambda $, and\newline
(2) $X=[span\{X_{\lambda }:\lambda \in \Lambda \}]_{\alpha }$.
\end{definition}

\begin{definition}
Let $X$ be a weak*-closed subspace of $\mathcal{M}$ and $\alpha \in
N_{\Delta }\left( \mathcal{M},\tau \right) $. Then $X$ is called an internal
column sum of a family of weak*-closed subspaces $\{X_{\lambda }\}_{\lambda
\in \Lambda }$ of $L^{\alpha }(\mathcal{M},\tau )$, denoted by $%
X=\bigoplus_{\lambda \in \Lambda }^{col}X_{\lambda }$ if\newline
(1) $X_{\mu }^{\ast }X_{\lambda }=\{0\}$ for all distinct $\lambda ,\mu \in
\Lambda $, and\newline
(2) $X=\overline{span\{X_{\lambda }:\lambda \in \Lambda \}}^{w\ast }$.
\end{definition}

\begin{lemma}
\label{colsum}\text{(from }\cite{B.L.}\text{)} Let $\mathcal{M}$ be a finite
von Neumann algebra with a faithful, normal, tracial state $\tau $ and $%
\alpha $ be a normalized, unitarily invariant $\Vert \cdot \Vert _{1,\tau }$%
-dominating continuous norm on $\mathcal{M}$. Let $H^{\infty }$ be a finite
subdiagonal subalgebra of $\mathcal{M}$ and $\mathcal{D}=H^{\infty }\cap
(H^{\infty })^{\ast }$. Assume that $\mathcal{W}\subseteq \mathcal{M}$ is a
weak*-closed subspace such that $\mathcal{W}H^{\infty }\subseteq \mathcal{W}$
. Then there exists a weak*-closed subspace $\mathcal{Y}$ of $\mathcal{M}$
and a family $\{u_{\lambda }\}_{\lambda \in \Lambda }$of partial isometries
in $\mathcal{M}$ such that\newline
(1) $u_{\lambda }^{\ast }\mathcal{Y}=0$ for all $\lambda \in \Lambda $,%
\newline
(2) $u_{\lambda }^{\ast }u_{\lambda }\in \mathcal{D}$ and $u_{\lambda
}^{\ast }u_{\mu }=0$ for all $\lambda ,\mu \in \Lambda $ with $\lambda \neq
\mu $,\newline
(3) $\mathcal{Y}=\overline{H_{0}^{\infty }\mathcal{Y}}^{w\ast }$,\newline
(4) $\mathcal{W}=\mathcal{Y}\oplus ^{col}(\oplus _{\lambda \in
\Lambda }^{col}u_{\lambda }H^{\infty })$.
\end{lemma}

Now we are ready to prove our main result of the paper, an extension of the
Chen-Hadwin-Shen theorem for noncommutative Hardy spaces associated with
finite von Neumann algebras.

\begin{theorem}
\label{main}Let $\mathcal{M}$ be a finite von Neumann algebra with a
faithful, normal, tracial state $\tau $ and $\alpha $ be a determinant,
normalized, unitarily invariant, continuous norm on $\mathcal{M}$. Then
there exists a faithful normal tracial state $\rho $ on $\mathcal{M}$ such
that $\alpha \in N_{1}\left( \mathcal{M},\rho \right) $. Let $H^{\infty }$
be a finite subdiagonal subalgebra of $\mathcal{M}$ and $\mathcal{D}%
=H^{\infty }\cap (H^{\infty })^{\ast }$. If $\mathcal{W}$ is a closed
subspace of $L^{\alpha }(\mathcal{M},\tau )$ such that $\mathcal{W}H^{\infty
}\subseteq \mathcal{W}$, then there exists a closed subspace $\mathcal{Y}$
of $L^{\alpha }(\mathcal{M},\tau )$ and a family $\{u_{\lambda }\}_{\lambda
\in \Lambda }$ of partial isometries in $\mathcal{M}$ such that\newline
(1) $u_{\lambda }^{\ast }\mathcal{Y}=0$ for all $\lambda \in \Lambda $,%
\newline
(2) $u_{\lambda }^{\ast }u_{\lambda }\in \mathcal{D}$ and $u_{\lambda
}^{\ast }u_{\mu }=0$ for all $\lambda ,\mu \in \Lambda $ with $\lambda \neq
\mu $,\newline
(3) $\mathcal{Y}=[H_{0}^{\infty }\mathcal{Y}]_{\alpha }$,\newline
(4) $\mathcal{W}=\mathcal{Y}\oplus ^{col}(\oplus _{\lambda \in
\Lambda }^{col}u_{\lambda }H^{\alpha }).$
\end{theorem}

\begin{proof}
The only if part is obvious. Suppose $\mathcal{W}$ is a closed subspace of $%
L^{\alpha }(\mathcal{M},\tau )$ such that $\mathcal{W}H^{\infty }\subset 
\mathcal{W}$. Then it follows from part(2) of the theorem \ref{density} that 
$\mathcal{W}\cap \mathcal{M}$ is weak* closed in $(\mathcal{M},\tau )=(%
\mathcal{M},\rho )$, we also notice $L^{\infty }(\mathcal{M},\tau )=\mathcal{%
M}=L^{\infty }(\mathcal{M},\rho ),L^{\alpha }(\mathcal{M},\tau )=L^{\alpha }(%
\mathcal{M},\rho )$ and $H^{\alpha }(\mathcal{M},\tau )=H^{\alpha }(\mathcal{%
M},\rho )$. It follows from the lemma \ref{colsum} that%
\begin{equation*}
\mathcal{W}\cap \mathcal{M}=\mathcal{Y_{1}}\bigoplus^{col%
}(\bigoplus_{i\in \mathcal{I}}^{col}u_{i}H^{\infty }),
\end{equation*}%
where $\mathcal{Y}_{1}$ is a closed subspace of $L^{\infty }(\mathcal{M}%
,\rho )$ such that $\mathcal{Y}_{1}=\overline{\mathcal{Y}_{1}H_{0}^{\infty }}%
^{w\ast }$, and where $u_{i}$ are partial isometries in $\mathcal{W}\cap 
\mathcal{M}$ with $u_{j}^{\ast }u_{i}=0$ if $i\neq j$ and with $u_{i}^{\ast
}u_{i}\in \mathcal{D}$. Moreover, for each $i,u_{i}^{\ast }\mathcal{Y}%
_{1}=\{0\}$, left multiplication by the $u_{i}u_{i}^{\ast }$ are contractive
projections from $\mathcal{W}\cap \mathcal{M}$ onto the summands $%
u_{i}H^{\infty }$, and left multiplication by $I-\sum_{i}u_{i}u_{i}^{\ast }$
is a contractive projection from $\mathcal{W}\cap \mathcal{M}$ onto $%
\mathcal{Y}_{1}$.

Let $\mathcal{Y}=[\mathcal{Y}_{1}]_{\alpha }$. It is not hard to verify that
for each $i,u_{i}^{\ast }\mathcal{M}=\{0\}$. We also claim that $%
[u_{i}H^{\infty }]_{\alpha }=u_{i}H^{\alpha }$. In fact it is obvious that $%
[u_{i}H^{\infty }]_{\alpha }\subseteq u_{i}H^{\alpha }$. We will need only
to show that $[u_{i}H^{\infty }]_{\alpha }\subseteq u_{i}H^{\alpha }$. Suppose $x\in \lbrack
u_{i}H^{\infty }]_{\alpha }$, there is a net $%
\{{x_{n}\}}_{n=1}^{\infty }\subseteq H^{\infty }$ such that $\alpha (u_{i}x_{n}-x)\rightarrow
0 $. By the choice of $u_{i},$ we know that $u_{i}u\in \mathcal{D}\subseteq
H^{\infty }$, so $u_{i}ux_{n}\in H^{\infty }$ for each $n\geq 1$. Combining
with the fact that $\alpha (u_{i}^{\ast }u_{i}x_{n}-u_{i}^{\ast }x)\leq
\alpha (u_{i}x_{n}-x)\rightarrow 0$, we obtain that $u_{i}^{\ast }x\in
H^{\alpha }$. Again from the choice of $u_{i}$, we know that $%
u_{i}u_{i}^{\ast }u_{i}x_{n}=u_{i}x_{n}$ for each $n\geq 1$. This implies
that $x=u_{i}(u_{i}^{\ast }x)\in u_{i}H^{\alpha }$. Thus we conclude that $%
[u_{i}H^{\infty }]_{\alpha }\subseteq u_{i}H^{\alpha }$, so $[u_{i}H^{\infty
}]_{\alpha }=u_{i}H^{\alpha }$. Now from parts (3) and (4) of the theorem %
\ref{density} and from the definition of internal column sum, it follows that%
\begin{align*}
\mathcal{W}& =[\mathcal{W}\cap \mathcal{M}]_{\alpha }=[\overline{span\{%
\mathcal{Y}_{1},u_{i}H^{\infty }:i\in \mathcal{I}\}}^{\ast }]_{\alpha
}=[span\{\mathcal{Y}_{1},u_{i}H^{\infty }:i\in \mathcal{I}\}]_{\alpha }. \\
& =[span\{\mathcal{Y},u_{i}H^{\alpha }:i\in \mathcal{I}\}]_{\alpha }=%
\mathcal{Y}\bigoplus^{col}(\bigoplus_{i\in \mathcal{I}}^{col%
}u_{i}H^{\alpha }).
\end{align*}%
\newline

Next, we will verify that $\mathcal{Y}=[\mathcal{Y}H_{0}^{\infty }]_{\alpha
} $. Recall that $\mathcal{Y}=[\mathcal{Y}_{1}]_{\alpha }$. It follows from
part (1) of the theorem \ref{density}, we have\newline
\begin{equation*}
\lbrack \mathcal{Y}_{1}H_{0}^{\infty }]_{\alpha }\cap \mathcal{M}=\overline{%
\mathcal{Y}_{1}H_{0}^{\infty }}^{w\ast }=\mathcal{Y}_{1}.
\end{equation*}%
Hence from part (3) of the theorem \ref{density} we have that\newline
\begin{equation*}
\mathcal{Y}\supseteq \lbrack \mathcal{Y}H_{0}^{\infty }]_{\alpha }\supseteq
\lbrack \mathcal{Y}_{1}H_{0}^{\infty }]_{\alpha }=[[\mathcal{Y}%
_{1}H_{0}^{\infty }]_{\alpha }\cap \mathcal{M}]_{\alpha }=[\mathcal{Y}%
_{1}]_{\alpha }=\mathcal{Y}.
\end{equation*}%
Thus $\mathcal{Y}=[\mathcal{Y}H_{0}^{\infty }]_{\alpha }$. Moreover, it is
not difficult to verify that for each $i$, left multiplication by the $%
u_{i}u_{i}^{\ast }$ are contractive projections from $\mathcal{K}$ onto the
summands $u_{i}H^{\alpha }$, and left multiplication by $I-%
\sum_{i}u_{i}u_{i}^{\ast }$ is a contractive projection from $\mathcal{W}$
onto $\mathcal{Y}$. Now the proof is completed.
\end{proof}

If we consider $\alpha $ as some specific norms, then we have some
corollaries. If we take $\alpha $ be a unitarily invariant, $\Vert \cdot
\Vert _{1,\tau }$-dominating, continuous norm, then we have
Chen-Hadwin-Shen's result in \cite{Chen}.

\begin{corollary}
Let $\mathcal{M}$ be a finite von Neumann algebra with a faithful, normal,
tracial state $\tau $ and $\alpha $ be a normalized, unitarily invariant, $%
\Vert \Vert _{1,\tau }$-dominating, continuous norm on $\mathcal{M}$. Let $%
H^{\infty }$ be a finite subdiagonal subalgebra of $\mathcal{M}$. Let $%
\mathcal{D}=H^{\infty }\cap (H^{\infty })^{\ast }$. Assume that $\mathcal{W}$
is a closed subspace of $L^{\alpha }(\mathcal{M},\tau )$ such that $%
H^{\infty }\mathcal{W}\subseteq \mathcal{W}$ . Then there exists a closed
subspace $\mathcal{Y}$ of $L^{\alpha }(\mathcal{M},\tau )$ and a family $%
\{u_{\lambda }\}_{\lambda \in \Lambda }$ of partial isometries in $\mathcal{W%
}\cap \mathcal{M}$ such that\newline
(1) $u_{\lambda }^{\ast }Y=0$ for all $\lambda \in \Lambda $,\newline
(2) $u_{\lambda }^{\ast }u_{\lambda }\in \mathcal{D}$ and $u_{\lambda
}^{\ast }u_{\mu }=0$ for all $\lambda ,\mu \in \Lambda $ with $\lambda \neq
\mu $,\newline
(3) $\mathcal{Y}=[H_{0}^{\infty }\mathcal{Y}]_{\alpha }$,\newline
(4) $\mathcal{W}=\mathcal{Y}\oplus ^{col}(\oplus _{\lambda \in
\Lambda }^{col}H^{\alpha }u_{\lambda })$.\newline
\end{corollary}

If we take $\alpha =\left\Vert \cdot \right\Vert _{p}$, then we have D.
Blecher and L. E. Labuschagne's result in \cite{B.L.}.

\begin{corollary}
Let $\mathcal{M}$ be a finite von Neumann algebra with a faithful, normal,
tracial state $\tau $ and $H^{\infty }$ be a finite subdiagonal subalgebra
of $\mathcal{M}$. Let $\mathcal{D}=H^{\infty }\cap (H^{\infty })^{\ast }$.
Assume that $\mathcal{W}$ is a closed subspace of $L^{p}(\mathcal{M},\tau
),1\leq p\leq \infty $ such that $H^{\infty }\mathcal{W}\subseteq \mathcal{W}
$ . Then there exists a closed subspace $\mathcal{Y}$ of $L^{p}(\mathcal{M}%
,\tau )$ and a family $\{u_{\lambda }\}_{\lambda \in \Lambda }$ of partial
isometries in $\mathcal{W}\cap \mathcal{M}$ such that\newline
(1) $u_{\lambda }^{\ast }\mathcal{Y}=0$ for all $\lambda \in \Lambda $,%
\newline
(2) $u_{\lambda }^{\ast }u_{\mu }\in \mathcal{D}$ and $u_{\lambda }^{\ast
}u_{\mu }=0$ for all $\lambda ,\mu \in \Lambda $ with $\lambda \neq \mu $,%
\newline
(3) $\mathcal{Y}=[H_{0}^{\infty }\mathcal{Y}]_{p}$,\newline
(4) $\mathcal{W}=\mathcal{Y}\oplus ^{col}(\oplus _{\lambda \in
\Lambda }^{col}H^{p}u_{\lambda })$.
\end{corollary}

\section{Generalized Beurling theorem for special von Neumann algebras}

In main theorem \ref{main}, if we take $\mathcal{M}=L^{\infty }(\mathbb{%
\mathbb{T}},\mu ),$ $H^{\infty }=H^{\infty }(\mathbb{T},\mu ),$ then $%
\mathcal{D=}H^{\infty }\cap (H^{\infty })^{\ast }=%
\mathbb{C}
$ and the center $\mathcal{Z}$ of $\mathcal{M}=L^{\infty }(\mathbb{T},\mu )$
is itself. So $\mathcal{Z}\nsubseteqq \mathcal{D=%
\mathbb{C}
}.$ However, for a finite von Neumann algebra $\mathcal{M}$ with a faithful
normal tracial state $\tau ,$ let $H^{\infty }$ be a finite subdiagonal
subalgebra of $\mathcal{M},$ $\mathcal{D=}H^{\infty }\cap (H^{\infty
})^{\ast },$ if the center $\mathcal{Z}\subseteq \mathcal{D},$ then
generalized Beurling theorem holds for normalized, unitarily invariant,
continuous norms on $(\mathcal{M},\tau ).$

\begin{theorem}
Let $\mathcal{M}$ be a finite von Neumann algebra with a faithful, normal,
tracial state $\tau .$ Let $H^{\infty }$ be a finite subdiagonal subalgebra
of $\mathcal{M},$ $\mathcal{D}=H^{\infty }\cap (H^{\infty })^{\ast }$, and
the center $\mathcal{Z}\subseteq \mathcal{D}.$ Let $\alpha $ be a
normalized, unitarily invariant, continuous norm on $(\mathcal{M},\tau )$.
If $\mathcal{W}$ is a closed subspace of $L^{\alpha }(\mathcal{M},\tau )$
such that $\mathcal{W}H^{\infty }\subseteq \mathcal{W}$, then there exists a
closed subspace $\mathcal{Y}$ of $L^{\alpha }(\mathcal{M},\tau )$ and a
family $\{u_{\lambda }\}_{\lambda \in \Lambda }$ of partial isometries in $%
\mathcal{M}$ such that\newline
(1) $u_{\lambda }^{\ast }\mathcal{Y}=0$ for all $\lambda \in \Lambda $,%
\newline
(2) $u_{\lambda }^{\ast }u_{\lambda }\in \mathcal{D}$ and $u_{\lambda
}^{\ast }u_{\mu }=0$ for all $\lambda ,\mu \in \Lambda $ with $\lambda \neq
\mu $,\newline
(3) $\mathcal{Y}=[H_{0}^{\infty }\mathcal{Y}]_{\alpha }$,\newline
(4) $\mathcal{W}=\mathcal{Y}\oplus ^{col}(\oplus _{\lambda \in
\Lambda }^{col}u_{\lambda }H^{\alpha })$.
\end{theorem}

\begin{proof}
By theorem \ref{cdominating}, there exists a faithful normal tracial state $%
\rho $ on $\mathcal{M}$ and a $c>0$ such that $\alpha $ is a continuous
normalized unitarily invariant $c\Vert \cdot \Vert _{1,\rho }$-dominating
norm on $(\mathcal{M},\rho )$. We know that $\Phi _{\mathcal{D},\tau }$ is
multiplicative on $H^{\infty }$. In general, $\Phi _{\mathcal{D},\rho }$
won't be multiplicative on $H^{\infty }$, however, the condition $\mathcal{Z}%
\subset \mathcal{D}$ makes sure $\Phi _{\mathcal{D},\rho }$ is
multiplicative on $H^{\infty }$. In order to do this,.we can choose $0\leq
x_{1}\leq x_{2}\leq \cdots $ in $\mathcal{Z}$ such that, for every $x\in 
\mathcal{M},$%
\begin{equation*}
\rho \left( x\right) =\lim_{n\rightarrow \infty }\tau \left( x_{n}x\right)
=\lim_{n\rightarrow \infty }\tau \left( \Phi _{\mathcal{D},\tau }\left(
x_{n}x\right) \right) .
\end{equation*}%
Since $\mathcal{Z\subset D}$, $\Phi _{\mathcal{D},\tau }\left( x_{n}x\right)
=x_{n}\Phi _{\mathcal{D},\tau }\left( x\right) .$ Thus%
\begin{equation*}
\rho \left( x\right) =\lim_{n\rightarrow \infty }\tau \left( x_{n}\Phi _{%
\mathcal{D},\tau }\left( x\right) \right) =\rho \left( \Phi _{\mathcal{D}%
,\tau }\left( x\right) \right) .
\end{equation*}%
It follows that $\Phi _{\mathcal{D},\tau }=\Phi _{\mathcal{D},\rho }$. This
now reduces to the $c\left\Vert \cdot \right\Vert _{1}$-dominating version
of the Chen-Hadwin-Shen theorem in \cite{Chen} above.
\end{proof}

\end{document}